\documentclass[a4paper,10pt]{article}
\usepackage{geometry}  
\usepackage{graphicx}
\usepackage{array}
\usepackage{amssymb}
\usepackage{amsmath}
\usepackage{amsthm}
\usepackage{graphicx}
\usepackage[parfill]{parskip} 
\usepackage[utf8]{inputenc}
\usepackage[english]{babel}
\usepackage{algorithm}
\usepackage{enumitem}
\usepackage{pgf}
\usepackage{tikz}
\usetikzlibrary{arrows,automata}
\usepackage[noend]{algpseudocode}
\usepackage{caption}
\usepackage[colorlinks=true,linkcolor=blue]{hyperref}
\usepackage{subcaption}
\usepackage{fancyhdr}
\usepackage{xfrac}
\usepackage{authblk}
\newcommand{\Q}{\mathbb{Q}}
\newcommand{\R}{\mathbb{R}}
\newcommand{\Z}{\mathbb{Z}}
\newcommand{\C}{\mathbb{C}}
\newcommand{\modb}[1]{\,\left(\mathrm{mod}\,{#1}\right)}

\newcommand{\conv}{\operatorname{conv}\!}

\theoremstyle{theorem}
\newtheorem{thm}{Theorem}[section]
\theoremstyle{definition}

\theoremstyle{theorem}
\newtheorem{prop}[thm]{Proposition}
\theoremstyle{definition}
\newtheorem{dfn}[thm]{Definition}
\theoremstyle{theorem}
\newtheorem{lem}[thm]{Lemma}
\theoremstyle{definition}
\newtheorem{ex}[thm]{Example}
\theoremstyle{definition}

\theoremstyle{definition}

\graphicspath{ {images/} }

\begin{document}

\title{Winding Number of $r$-modular sequences and Applications to the Singularity Content of a Fano Polygon
\thanks{MSC2010: Primary: 14M25; Secondary: 05A99. 14B05, 14J45. 
Keywords: Fano polygon, cyclic quotient singularity, singularity content, $r$-modular sequence, winding number.}}
\author{Daniel Cavey\thanks{School of Mathematical Sciences, University of Nottingham, Nottingham, NG7 2RD, UK. \\
Email address: Daniel.Cavey@nottingham.ac.uk}
\;and Akihiro Higashitani\thanks{Department of Mathematics, Kyoto Sangyo University, 603-8555, Kyoto, Japan. \\
Email address: ahigashi@cc.kyoto-su.ac.jp}}
\date{}

\maketitle

\begin{abstract}
By generalising the notion of a unimodular sequence, we create an expression for the winding number of certain ordered sets of lattice points. Since the winding number of the vertices of a Fano polygon is necessarily one, we use this expression as a restriction to classify all Fano polygons without T-singularities and whose basket of residual singularities is of the form $\left\{ \frac{1}{r}(1,s_{1}), \frac{1}{r}(1,s_{2}), \ldots, \frac{1}{r}(1,s_{k}) \right\}$ for $k,r \in \Z_{>0}$, and $1 \leq s_{i} < r$ is coprime to $r$.
\end{abstract}

\section{Introduction}

Mirror symmetry has provided a new approach to the classification of del Pezzo surfaces, and motivated many combinatorial problems. Specifically it is conjectured in \cite{MirrorSymmetryandtheClassificationofOrbifoldDelPezzoSurfaces} that qG-deformation equivalence classes of locally qG-rigid class TG orbifold del Pezzo surfaces are in one-to-one correspondence with mutation equivalence classes of Fano polygons. The singularities of the orbifold surfaces can be read at the combinatorial level of the Fano polygon, and we exploit these combinatorics to deduce statements about orbifold del Pezzo surfaces that have particular singularities. The singularities we are interested in are R-singularities which are characterised among cyclic quotient singularities, see Section \ref{2} for formal definitions of these terms, as not admitting a smoothing via a qG-deformation. It follows that the mutation equivalence class of the Fano polygon has cardinality one. Therefore we do not need to concern ourselves with the notion of a combinatorial mutation for this paper. The main results of this paper translated from the language of a classification of Fano polygons into the language of orbifold del Pezzo surfaces are as follows:

\begin{thm}
The classification of qG-rigid orbifold del Pezzo surfaces that admit a toric degeneration, have topological Euler number $0$ and have singular locus equal to a collection of isolated points $p_{i} = \frac{1}{r}(1,s_{i})$ has been completed. Each surface is described as the toric variety $X_{P}$ corresponding to a Fano polygon $P$ with vertex set $\mathcal{V}(P)$ as listed in Theorem \ref{1.6}.
\end{thm}

\begin{thm}
The existence of a qG-rigid orbifold del Pezzo surface that admits a toric degeneration, has topological Euler number $0$ and has singular locus equal to a collection of isolated points $\left\{ k \times \frac{1}{r}(1,s) \right\}$ is understood in terms of necessary and sufficient conditions on $k,r,s$. These conditions are listed in Theorem \ref{1.7}.
\end{thm}

\section{Cyclic Quotient Singularities and Singularity Content}\label{2}

For toric surfaces, which can be studied in Cox--Little--Schenck \cite{ToricVarieties} or Fulton \cite{IntroductiontoToricVarieites}, the singularities with which one is concerned are known as \emph{cyclic quotient singularities}, which we now describe. Throughout the paper we work in the lattice $N=\Z^{2}$. By a $GL(N)$-transformation, any cone $\sigma \subset N_{\R} = N \otimes_{\Z} \R $ can be assumed to be of the form $\text{span}_{\R_{\geq 0}} \left( \left( 0,1 \right), \left( r, -s \right) \right)$ where $0 \leq s < r$ and $\gcd(r,s)=1$. Representing these primitive ray generators by $e_{2}$ and $e_{1}^{r}e_{2}^{-s}$ respectively, it follows that the dual cone $\sigma^{\vee} \subset M = \text{Hom}(N,\Z)$ has generating rays $e_{1}^{*}$ and $e_{1}^{*s}e_{2}^{*r}$. Define the semigroup $S_{\sigma} = \sigma^{\vee} \cap M$. Take the affine ring over this semigroup:
\[ \C[S_{\sigma}] = \underset{(i,j) \text{ where } j \leq \frac{r}{s}i}{\oplus} \C e_{1}^{*i} e_{2}^{*j} \] 
and use it to define the corresponding affine patch $U_{\sigma} = \text{Spec} \left( \C \left[ S_{\sigma} \right] \right)$.

It has been shown that this affine patch is given by $U_{\sigma} = \C^{2} / \mu_{r}$ where $\mu_{r}$ is the cyclic group of order $r$, and the action of a primitive $r^{th}$ root of unity $\epsilon$ is via
\[ \epsilon \cdot \left( x,y \right) = \left( \epsilon x, \epsilon^{s} y \right). \]
The germ at the origin of $U_{\sigma}$ is the cyclic quotient singularity denoted $\frac{1}{r}(1,s)$. 
It is with this description in mind that throughout the paper we somewhat blur the distinction between a cone and the corresponding cyclic quotient singularity. In the sequel, we will use the notation $\sigma$ for a cyclic quotient singularity and $C_\sigma$ for the corresponding cone. 

Consider an arbitrary cone $C_{\sigma}$ with primitive ray generators $\rho_{1}$ and $\rho_{2}$. Let $H$ be the unique hyperplane through $\rho_{1}$ and $\rho_{2}$, and set $E = C_{\sigma} \cap H$. For $C_{\sigma}$, there are well defined notions of \emph{lattice length} $\ell (C_{\sigma})= \lvert E \cap N \rvert -1$, and \emph{lattice height} $h(C_{\sigma}) = \lvert \langle E, n_{E} \rangle \rvert$ where $n_{E}\in M$ is the unique primitive inward pointing normal of $E$.

\begin{dfn}[\cite{ThreefoldsandDeformationsofSurfaceSingularities}]
A cyclic quotient singularity $\sigma$ is a \emph{T-singularity} if $h(C_{\sigma}) \mid \ell (C_{\sigma})$ in which case we call $C_{\sigma}$ a T-cone. If $h(C_{\sigma}) = \ell (C_{\sigma})$, then $\sigma$ is a \emph{primitive} T-singularity. A cyclic quotient singularity $\sigma$ is a \emph{R-singularity} if $\ell (C_{\sigma}) < h(C_{\sigma})$ and then $C_{\sigma}$ an R-cone.
\end{dfn}

\begin{ex}
Consider the cone $C_{\frac{1}{2}(1,1)}$ which has primitive ray generators $(0,1)$ and $(2,-1)$. Calculate $\ell \left( C_{\frac{1}{2}(1,1)} \right) =2$ and $h \left( C_{\frac{1}{2}(1,1)} \right)=1$. Therefore a $\frac{1}{2}(1,1)$ cyclic quotient singularity is a (non-primitive) T-singularity. Alternatively the cone $C_{\frac{1}{3}(1,1)}$, described by primitive ray generators $(0,1)$ and $(3,-1)$, satisfies $\ell \left( C_{\frac{1}{3}(1,1)} \right) =1$ and $h \left( C_{\frac{1}{3}(1,1)} \right)=3$. The singularity $\frac{1}{3}(1,1)$ is an R-singularity.
\end{ex}

The deformation theory of cyclic quotient singularities has been studied by Altmann--Christopherson--Ilten--Stevens \cite{TheVersalDeformationofanIsolatedToricGorensteinSingularity,P-resolutionsofCyclicQuotientsfromtheToricViewpoint,OntheComponentsandDiscriminantoftheVersalBaseSpaceofCyclicQuoitentSingularities,OneparameterToricDeformationsofCyclicQuotientSingularities,OntheVersalDeformationofCyclicQuotientSingularities}. A \emph{qG-deformation} is a deformation of the toric surface that preserves the numerics of the anti-canonical divisor. T-singularities on the the surface are smoothable via a qG-deformation, whereas R-singularities are not smoothable under a qG-deformation.

A cone $C$ that defines neither a T-singularity nor an R-singularity, that is $\ell \left( C \right) = n h \left( C \right) + r$ where $n \in \Z_{>0}$ and $0< r < h \left( C \right)$, can be sub-divided into cones $C_{0}, C_{1}, \ldots, C_{n}$ where
\[ \ell \left( C_{0} \right) = r, \qquad \text{ and } \qquad \ell \left( C_{i} \right)  = h \left( C \right), \text{ for } i \in \{ 1, \ldots, n \}. \]
Then $C_{0}$ is an R-cone and so corresponds to an R-singularity which we denote $\text{res} \left( \sigma \right)$, while $C_{1},\ldots,C_{n}$ are T-cones. Akhtar--Kasprzyk \cite{SingularityContent} define the \emph{singularity content} of $C$ to be the pair $\left( n, \text{res} \left( \sigma \right) \right)$.

\begin{dfn}[{\cite[Definition 3.1]{SingularityContent}}]
Let $P \subset N_{\R}$ be a Fano polygon. Label the edges of $P$ clockwise $E_{1},\ldots ,E_{k}$. Let $C_{\sigma_{i}}$ be the cone over the edge $E_{i}$. Set 
\[ \mathrm{SC} \left( \sigma_{i} \right) = \left( n_{i},\mathrm{res} \left( \sigma_{i} \right) \right). \] 
Define the \emph{singularity content} of $P$ to be:
\[ \mathrm{SC}(P) = \left( \sum\limits_{i=1}^{k} n_{i},\mathcal{B} \right), \]
where $\mathcal{B} = \left\{ \mathrm{res} \left( \sigma_{1} \right) , \ldots, \mathrm{res} \left( \sigma_{k} \right) \right\}$ is a cyclically ordered set known as the \emph{basket of residual singularities}.
\end{dfn}

It is an interesting problem to complete classifications of Fano polygons by singularity content, and some such results are available \cite{RestrictionsOnTheSingularityContentOfAFanoPolygon,ClassificationofPolygonswithGivenSingularityContent,MinimalityandMutationEquivalenceofPolygons}. In this vain the first main result of this paper is as follows:

\begin{thm}\label{1.6}
Let $r \in \mathbb{Z}_{>0} \backslash \{1,2,4\}$.  Any Fano polygon $P$ with singularity content
\[ \text{SC}(P) = \left( 0, \left\{ \frac{1}{r}(1,s_{1}), \frac{1}{r}(1,s_{2}), \ldots, \frac{1}{r}(1,s_{k}) \right\} \right) \]
has $k\in \{ 3,4,5,6\}$ and vertex set unimodular equivalent to one of the following:
\begin{itemize}
\item $\left\{ (0,1), (-r,s-1), (r,-s) \right\}$, where $\gcd (r,s) = \gcd (r,s-1) =1$;
\item $\left\{ (0,1), (-r,s), (0,-1), (r,-s) \right\}$, where $\gcd (r,s) =1$;
\item $\left\{ (0,1), (-r,s+1), (0,-1), (r,-s) \right\}$, where $\gcd (r,s) = \gcd (r,s+1) =1$;
\item $\left\{ (0,1), (-r,s), (r,-s-1), (r,-s) \right\}$, where $\gcd (r,s) = \gcd (r,s+1) =1$;
\item $\left\{ (0,1), (-r,s), (-r,s-1), (r,-s) \right\}$, where $\gcd (r,s) = \gcd (r,s-1) =1$;
\item $\left\{ (0,1), (-r,s+1), (-r,s), (0,-1), (r,-s) \right\}$, where $\gcd (r,s) = \gcd (r,s+1) =1$;
\item $\left\{ (0,1), (-r,s), (-r,s-1), (0,-1), (r,-s) \right\}$, where $\gcd (r,s) = \gcd(r,s-1) =1$;
\item $\left\{ (0,1), (-r,s+1), (-r,s), (r,-s-1), (r,-s) \right\}$, where $\gcd (r,s) = \gcd (r,s+1) =1$;
\item $\left\{ (0,1), (-r,s+1), (-r,s),  (0,-1), (r,-s-1), (r,-s) \right\}$, where $\gcd (r,s) = \gcd (r,s+1) =1$.
\end{itemize}
\end{thm}

It is worth noting that this theorem is a one way implication. All Fano polygons with this residual basket will fall into one of the models listed, however not all of the models will describe a Fano polygon with the desired singularity content. For example consider the first model on four vertices: when $s=1$ this polygon will have T-cones and so the singularity content is not of the form stated in the theorem. The explicit cyclic quotient singularities for each cone are described in the figures of Section \ref{3}.

In Section \ref{Section4} we consider each family occurring in Theorem \ref{1.6}, and further set the requirement that $\frac{1}{r}(1,s_{1}) = \frac{1}{r}(1,s_{2}) = \ldots = \frac{1}{r}(1,s_{k})$ to obtain the second main result:

\begin{thm}\label{1.7}
There exists a Fano polygon $P$ such that 
\[ \text{SC}(P) = \left( 0, \left\{ k \times \frac{1}{r}(1,s) \right\} \right), \]
if and only if one of the following holds:
\begin{itemize}
\item $k=3$, $p \equiv 1 \modb{6}$ for all primes $p \mid r$, and $s^{2}-s+1 \equiv 0 \modb{r}$;
\item $k=4$, $p \equiv 1 \modb{4}$ for all primes $p \mid r$, and $s^{2} + 1 \equiv 0 \modb{r}$;
\item$k=6$, $r=3$ and $s=1$;
\item $k=6$, $p \equiv 1 \modb{6}$ for all primes $p \mid r$, and $s^{2}+s+1 \equiv 0 \modb{r}$.
\end{itemize}
Furthermore in each of these cases $P$ is unique up to isomorphism with the exception of the case $(k,r,s)=(4,5,2)$ in which there are two non-isomorphic models for $P$.
\end{thm}

This is a stronger result than Theorem 1.1 of \cite{RestrictionsOnTheSingularityContentOfAFanoPolygon}. 

We recall a similar notion to $r$-modular sequence arising from a Fano polygon. 
\begin{dfn}[{\cite[Definition 1.1]{ReflexivePolytopesofHigherIndexandtheNumber12}}]
A Fano polygon $P$ is $\ell$-reflexive if every edge is of height $\ell$.
\end{dfn}

Note that the Fano polygons of Theorem \ref{1.7} are $\ell$-reflexive; 
Kasprzyk--Nill \cite{ReflexivePolytopesofHigherIndexandtheNumber12} also provide a restriction on the number of vertices of $\ell$-reflexive polygons and a classification of $\ell$-reflexive polygons.

\section{$r$-modular sequences}

The aim of this section is to generalise the formula of the winding number of a unimodular sequence and ``twelve-point theorem'' proved in \cite{LatticeMultiPolygons}. Indeed the main result in this section is a generalisation of Theorem 1.2 and Theorem 2.3 from this paper. The twelve arising here is expertly analysed by Poonen -- Rodriguez-Villegas\cite{LatticePolygonsandNumber12}.

We say that a lattice point $v \in \Z^2$ is \textit{primitive} if there is no lattice point in the line segment whose endpoints are the origin ${\bf 0}$ and $v$ except for the endpoints. 
\begin{dfn}
A sequence of vectors $v_{1},\ldots,v_{k}$, where each $v_{i} \in \Z^{2}$ is primitive, is said to be \emph{$r$-modular} if each parallelogram $\conv \left\{ \mathbf{0}, v_{i}, v_{i+1}, v_{i}+v_{i+1} \right\}$ contains exactly $r-1$ lattice points in its interior for $i \in \left\{ 1,\ldots, k \right\}$, where $v_{k+1}=v_{1}$. 
\end{dfn}
Note that the case $r=1$ is nothing but the notion of unimodular sequences. 
Indeed this definition is equivalent to $\det \begin{pmatrix} v_{i} & v_{i+1} \end{pmatrix} = \pm r, \forall i \in \left\{ 1,\ldots,k \right\}$. In this vain set $\epsilon_{i} = \frac{1}{r}\det \begin{pmatrix} v_{i} & v_{i+1} \end{pmatrix}$. 
This indicates whether the sequence is moving in an anticlockwise or clockwise direction.

As in the case of unimodular sequences, for each successive pair of vectors $(v_{i}, v_{i+1})$, there exists a matrix $M$ such that
\[ \begin{pmatrix} v_{i} & v_{i+1} \end{pmatrix} = \begin{pmatrix} v_{i-1} & v_{i} \end{pmatrix} M. \]
Necessarily $\det(M) = \epsilon_{i-1}\epsilon_i$, and so $M$ takes the form $M = \begin{pmatrix} 0 & -\epsilon_{i-1}\epsilon_i \\ 1 & -\epsilon_ia_i \end{pmatrix}$ for some $a_{i} \in \Q$. 
In fact, since each $v_i$ is primitive, by taking an appropriate unimodular matrix $H$, we can obtain that $H \begin{pmatrix} v_{i-1} & v_i \end{pmatrix} = \begin{pmatrix} r & 0 \\ -s &1 \end{pmatrix}$ for some $s \in \Z$. 
Then \[ H \begin{pmatrix} v_i & v_{i+1} \end{pmatrix} = H \begin{pmatrix} v_{i-1} & v_i \end{pmatrix} M = \begin{pmatrix} r & 0 \\ -s &1 \end{pmatrix}\begin{pmatrix} 0 & -\epsilon_{i-1}\epsilon_i \\ 1 &-\epsilon_ia_i \end{pmatrix}=\begin{pmatrix} 0 & -\epsilon_{i-1}\epsilon_ir \\ 1 & \epsilon_{i-1}\epsilon_is-\epsilon_ia_i \end{pmatrix} \in \Z^{2 \times 2}.\]
Since $\epsilon_{i-1},\epsilon_i \in \{\pm 1\}$ we conclude that $a_i \in \Z$. 

The definition of $a_i$ is equivalent to 
\begin{align}\label{aaa}
\epsilon_{i-1}v_{i-1}+\epsilon_iv_{i+1}+a_iv_i= \begin{pmatrix} 0 \\ 0 \end{pmatrix}. 
\end{align}

We use a general version of a lemma in \cite{LatticeMultiPolygons}. The proof of this generalised statement is identical to that of the original proof.

\begin{lem}[{\cite[Lemma 1.3]{LatticeMultiPolygons}}]\label{2.2}
Consider an $r$-modular sequence $v_{1}, \ldots, v_{k}$, and let $v_{j}$ be the vector among the sequence with maximal Euclidean norm. Then $a_{j} \in \left\{ 0, \pm 1 \right\}$.
\end{lem}

\begin{thm}\label{2.3}
Given an $r$-modular sequence $v_{1},\ldots, v_{k}$ where $k \geq 2$, its winding number is
\[ \frac{1}{12} \left( \sum\limits_{i=1}^{k} a_{i} + 3 \sum\limits_{i=1}^{k} \epsilon_{i} \right). \]
\end{thm}

\begin{proof}
The proof uses induction on the length of the $r$-modular sequence $k$. 

For the base case $k=2$, it is easy to see that $\epsilon_{1} = -\epsilon_{2}$, and so $a_{1}=a_{2}=0$ from the identity $-a_{i}v_{i} = \epsilon_{i-1}v_{i-1} + \epsilon_{i}v_{i+1}$, and that the winding number is $0$. The identity holds trivially.

For $k=3$ by an orientation preserving unimodular transformation assume $\left( v_{1}, v_{2} \right)$ is equal to either $\left( \left( r, -s \right), \left( 0,1 \right) \right)$ or $\left( \left( 0, 1 \right), \left( r,-s \right) \right)$ where $1 \leq s < r$ and $\gcd(r,s)=1$. In both these cases since $\det \begin{pmatrix} v_{2} & v_{3} \end{pmatrix} = \det  \begin{pmatrix} v_{3} & v_{1} \end{pmatrix}= \pm r$, necessarily $v_{3}$ is given by one of $(r,  1-s), (r,  -1-s), (-r,  1+s)$ or $(-r,  s-1)$. The formula can be routinely checked for each possibility.

Suppose $k \geq 4$, and that by inductive assumption all $r$-modular sequences with less than $k$ vertices satisfies the desired identity. Now choose $v_{j}$ to be the vertex with maximal Euclidean norm. By Lemma \ref{2.2} we know $a_{j} \in \{ 0, \pm 1 \}$. The inductive step is split into cases based on the value of $a_{j}$ and the proof follows exactly as in \cite[Theorem 1.2]{LatticeMultiPolygons}.
\end{proof}

Can this statement be generalised to any sequence of integer points, that is, allowing the value $\det  \begin{pmatrix} v_{i} & v_{i+1} \end{pmatrix}$ to be arbitrary? A major obstacle is that the identity $-a_{i}v_{i} = \epsilon_{i-1}v_{i-1} + \epsilon_{i}v_{i+1}$ would then contain an extra variable arising from a ratio between determinants. We then have a less strict bound available on the value $a_{j}$ where $j$ is the subscript indicating the integer point of maximal Euclidean norm, and furthermore the $a_{i}$ may become non-integer.

For an $r$-modular sequence $v_1,\ldots,v_k$, we set 
$$w_i=\frac{v_i-v_{i-1}}{\det \begin{pmatrix} v_{i-1} & v_i \end{pmatrix}} \;\; \text{for} \;\;i=1,\ldots,k,$$
where $v_0=v_k$. We remark that $w_i$ is not necessarily a lattice point. Given $P$ with vertices $v_{1},\ldots,v_{k}$, define $P^{\vee} = \left( w_{1}, \ldots, w_{k} \right)$.

For a sequence $P=(v_1,\ldots,v_k)$, let 
\[ B(P)=\sum_{i=1}^k \det\begin{pmatrix} v_i &v_{i+1} \end{pmatrix}. \]

Consider a sequence $P=(v_1,\ldots,v_k)$ of the vertices of a Fano polygon ordered anticlockwise. Then $\conv \left\{w_1,\ldots,w_k \right\}$ is 90 degree rotation of a polar dual of $P$. Also $B(P)$ and $B(P^\vee)$ coincide with their numbers of lattice points contained in its boundary respectively. 

\begin{thm}[Further Generalisation of Twelve-Point Theorem]
Let $P$ be an $r$-modular sequence and let $w(P)$ be its winding number. Then we have that 
\[ \frac{1}{r} \cdot B(P)+r \cdot B(P^\vee)=12\cdot w(P). \]
\end{thm}
\begin{proof}
By definition, we have \begin{align*} B(P)=\sum_{i=1}^k \det\begin{pmatrix} v_{i-1} &v_i \end{pmatrix}=r\sum_{i=1}^k \epsilon_i.\end{align*} 
On the other hand, it follows from the definition of $w_i$ and \eqref{aaa} that 
\begin{align*}
\det\begin{pmatrix} w_i &w_{i+1} \end{pmatrix}&=\frac{1}{\epsilon_{i-1}\epsilon_ir^2}\det\begin{pmatrix} v_i-v_{i-1} &v_{i+1}-v_i \end{pmatrix} \\
&=\frac{1}{\epsilon_{i-1}\epsilon_ir^2}\det\begin{pmatrix} v_i-v_{i-1} &(-\epsilon_ia_i-1)v_i-\epsilon_{i-1}\epsilon_iv_{i-1} \end{pmatrix} \\
&=\frac{1}{\epsilon_{i-1}\epsilon_ir^2}\det\begin{pmatrix} v_i-v_{i-1} &(-\epsilon_ia_i-1-\epsilon_{i-1}\epsilon_i)v_i \end{pmatrix} \\
&=\frac{\epsilon_ia_i+1+\epsilon_{i-1}\epsilon_i}{\epsilon_{i-1}\epsilon_ir^2}\det\begin{pmatrix} v_{i-1} &v_i \end{pmatrix} \\
&=\frac{a_i+\epsilon_i+\epsilon_{i-1}}{r}. 
\end{align*}
Hence, by Theorem \ref{2.3}, we obtain that 
\begin{align*}
\frac{1}{r} \cdot B(P)+r \cdot B(P^\vee) &= \sum_{i=1}^k \epsilon_i + r \sum_{i=1}^k \det\begin{pmatrix} w_{i-1} &w_i \end{pmatrix} \\
&=\sum_{i=1}^k \epsilon_i+\sum_{i=1}^k(a_{i-1}+\epsilon_{i-1}+\epsilon_{i-2}) \\
&=\sum_{i=1}^k a_i + 3\sum_{i=1}^k \epsilon_i \\
&=12 \cdot w(P), 
\end{align*}
as required. 
\end{proof}

\section{Fano polygons with determinant $r$ Cones}\label{3}

We seek to provide a classification of Fano polygons consisting of cones $\left\{ C_{i} \right\}_{i=1}^{k}$ such that each $C_{i}$ represents a $\frac{1}{r}(1,s_{i})$ cyclic quotient singularity where $r$ is some fixed positive integer and $1 \leq s_{i} < r$ with $\gcd(r,s_{i})=1$, for all $i$. The key observation here is that given such a Fano polygon $P$, then the anticlockwise ordered set of vertices $\mathcal{V}(P) = \left\{v_{1},\ldots, v_{k} \right\}$ forms an associated $r$-modular sequence with winding number $1$. Therefore Proposition \ref{2.3} provides a condition that the vertices of the Fano polygon must satisfy.

Furthermore properties of Fano polygons translate to properties of the associated $r$-modular sequence:

\begin{lem}
The $r$-modular sequence associated to a Fano polygon has the property $\epsilon_{i} =1 , \forall i$.
\end{lem}

\begin{proof}
This is trivial since by definition the vertices are traversed in an anticlockwise fashion.
\end{proof}

\begin{lem}\label{3.1}
The $r$-modular sequence associated to a Fano polygon has the property $a_{i} \geq -2, \forall i$. Furthermore the case $a_{i} =-2$ means that cones $C_{i}$ and $C_{i+1}$ share an edge of the Fano polygon, that is $v_{i+1}$ does not necessarily need to be listed as a vertex.
\end{lem}

\begin{proof}
Consider three arbitrary points $v_{i}, v_{i+1}$ and $v_{i+2}$ of the $r$-modular sequence. Without loss of generality assume that
\[ v_{1} = \begin{pmatrix} r \\ -s \end{pmatrix}, \qquad \text{ and } \qquad v_{2} = \begin{pmatrix} 0 \\ 1 \end{pmatrix}. \]
It follows that
\[ v_{3} = \begin{pmatrix} r & 0 \\ -s & 1 \end{pmatrix} \begin{pmatrix} -1 \\ -a_{i} \end{pmatrix} = \begin{pmatrix} -r \\ s-a_{i} \end{pmatrix}. \]
However by the convexity of the Fano polygon, we require that $s+2 \geq s-a_{i}$ from which the result follows.
\end{proof}

In the case where each cone represents a $\frac{1}{r}(1,1)$ R-singularity, it is derived in \cite{RestrictionsOnTheSingularityContentOfAFanoPolygon} that $k$ must be a multiple of $2r$. Furthermore since all cones of $P$ are R-cones and so cannot share an edge with another cone, each $a_{i}> -2$ by Lemma \ref{3.1}. So for some $l \in \Z_{>0}$:
\[ 12 = \sum\limits_{i=1}^{2rl} a_{i} + 3 \sum\limits_{i=1}^{2rl} \epsilon_{i} \geq \sum\limits_{i=1}^{2rl} (-1) + 3 \sum\limits_{i=1}^{2rl} (1) \geq -2rl + 6rl = 4rl. \]
Therefore:
\[ r \leq \frac{3}{l}. \]
This is perhaps a simpler way to complete the proof of Theorem 1.7 in \cite{RestrictionsOnTheSingularityContentOfAFanoPolygon}.

Note further that the winding number of the $r$-modular sequence associated to a Fano polygon $P$ is enough to provide a statement on $\# \mathcal{V} (P)$. Assume $P$ is a Fano polygon with $\mathcal{V}(P)= \left\{ v_{1}, \ldots, v_{k} \right\}$ where all the $v_{i}$ are necessary and the determinant of each cone is $r$. Then
\[ 12 = \sum\limits_{i=1}^{k} a_{i} + 3 \sum\limits_{i=1}^{k} \epsilon_{i} \geq \sum\limits_{i=1}^{k} (-1) + 3 \sum\limits_{i=1}^{k} (1) \geq -k + 3k = 2k.  \]
So the maximum number of vertices for a Fano polygon is 6. By investigating case by case for each value $k \in \left\{ 3,4,5,6 \right\}$, we construct all Fano polygons satisfying the necessary conditions allowing us to prove Theorem \ref{1.6}. 

In the sequel, for an arbitrary Fano polygon, we use the notation $\sigma_{i} = \text{Cone}(v_{i},v_{i+1})$.

\subsection{Case $k=3$}\label{SSec3.1}

By a $\text{GL}_{2}(N)$-transformation assume $v_{1} = \begin{pmatrix} r \\ -s \end{pmatrix}$ where $1 \leq s < r$ with $\gcd (r,s)=1$, and $v_{2}= \begin{pmatrix} 0 \\ 1 \end{pmatrix}$. It follows that $v_{3} = \begin{pmatrix} a \\ b \end{pmatrix}$ satisfies:
\[ \det \begin{pmatrix} 0 & a \\ 1 & b \end{pmatrix} = r, \qquad \implies \qquad a= -r, \]
and
\[ \det \begin{pmatrix} -r & r \\ b & -s \end{pmatrix} = r, \qquad \implies \qquad b= s-1. \]

So $v_{3}= \begin{pmatrix} -r \\ s-1 \end{pmatrix}$, meaning $\gcd(r,s-1)=1$ is also required. Therefore this describes a unique model for $k=3$.

\begin{figure}[H]
\centering
\begin{tabular}{| p{30pt} | p{133pt} | p{75pt} | p{152pt} |}
\hline
 & \begin{center} Polygon Model \end{center} & \begin{center} Conditions on variables \end{center} & \begin{center} Cyclic quotient singularities \end{center}\\ \hline \hline
 \begin{center} \vspace{35pt} Family 1 \end{center} & \vspace{0pt} \begin{tikzpicture}[scale=0.4, transform shape]
\begin{scope}
\clip (-6,-3.9) rectangle (6.3cm,2.8cm); 
\filldraw[fill=cyan, draw=blue] (5,-3) -- (0,1) -- (-5,2) -- (5,-3); 
\foreach \x in {-5,-4,...,5}{                           
    \foreach \y in {-7,-6,...,7}{                       
    \node[draw,shape = circle,inner sep=1pt,fill] at (\x,\y) {}; 
    }
}
 \node[draw,shape = circle,inner sep=4pt] at (0,0) {}; 
 \node at (-5,2.5) {$(-r,s-1)$};
 \node at (0,1.5) {$(0,1)$};
 \node at (5,-3.5) {$(r,-s)$};
 \node at (2.5,-0.5) {$\sigma_{1}$};
 \node at (-2,1.7) {$\sigma_{2}$};
 \node at (-0.7,-0.7) {$\sigma_{3}$};
\end{scope}
\end{tikzpicture} & \vspace{25pt} \begin{center} \begin{tabular}{c} $\gcd(r,s)=1$ \\ $\gcd(r,s-1)=1$ \end{tabular} \end{center} & \vspace{20pt} \begin{center} \begin{tabular}{c} $\sigma_{1} = \frac{1}{r}(1,s)$ \\ $\sigma_{2} = \frac{1}{r}(1,r+1-s)$ \\  $\sigma_{3} = \frac{1}{r} \left( r-(s-1)\left\lfloor \frac{r}{s} \right\rfloor, r - s \left\lfloor \frac{r}{s} \right\rfloor \right)$ \end{tabular} \end{center} \\ \hline
\end{tabular}
\caption{Unique family of Fano polygons with only determinant $r$ cones on 3 vertices.}\label{Fig1}
\end{figure}

It is worth noting at this point that not all the Fano polygons that arise here are $\ell$-reflexive polygons for some $\ell \in \Z$. Indeed if $(r,s)=(35,12)$ for family 1, the polygon consists of 3 R-singularities $\frac{1}{35}(1,3), \frac{1}{35}(1,17)$ and $\frac{1}{35}(1,19)$, however two edges are of height $35$ and the other of height $7$.

\subsection{Case $k=4$}

Assume without loss of generality that
\[ v_{1} = \begin{pmatrix} r \\ -s \end{pmatrix}, \qquad v_{2}= \begin{pmatrix} 0 \\ 1 \end{pmatrix}, \]
where $1 \leq s < r$ and $\gcd (r,s)=1$. As before in subsection \ref{SSec3.1}, if $v_{3} = \begin{pmatrix} a \\ b \end{pmatrix}$, then $\det (v_{2}, v_{3}) =r$ implies that $a=-r$. Let $v_{4}= \begin{pmatrix} c \\ d \end{pmatrix}$, and by $\det(v_{4},v_{1}) = r$ deduce that $d=-1 -\frac{cs}{r}$. Finally $v_{3}$ and $v_{4}$ are subject to the condition
\[ \det \begin{pmatrix} -r & c \\ b & -\frac{cs}{r}-1 \end{pmatrix} = r, \qquad \implies \qquad c(s-b)=0.  \]
Therefore either $c=0$ or $s=b$.

Suppose $c=0$. In this case the vertices of $P$ are given by:
\[ v_{1} = \begin{pmatrix} r \\ -s \end{pmatrix}, \qquad v_{2}= \begin{pmatrix} 0 \\ 1 \end{pmatrix}, \qquad v_{3}= \begin{pmatrix} -r \\ \alpha \end{pmatrix}, \qquad v_{4}= \begin{pmatrix} 0 \\ -1 \end{pmatrix},\]
where $\gcd(r,s)=\gcd (r,\alpha)=1$. The final condition that needs checked is that the vertices satisfy convexity. This is easily checked in the language of $r$-modular sequences by calculating
\[ a_{1}=0, \qquad a_{2} = s-\alpha, \qquad a_{3}=0, \qquad \text{ and } \qquad a_{4} = \alpha -s. \]
By Lemma \ref{3.1} convexity is equivalent to $a_{i}>-2, \forall i$, and so imposing the condition $\lvert s - \alpha \rvert < 2$ arising from both $a_{2},a_{4}>-2$ is enough. Note the case $\alpha = s-1$ is isomorphic to the case $\alpha=s+1$ by reflection in both axes and relabelling. The cases $s=\alpha$ and $s=\alpha +1$ provide the first two families shown in Figure \ref{Fig2}.

Now suppose $s=b$. The vertices of $P$ are of the form:
\[ v_{1} = \begin{pmatrix} r \\ -s \end{pmatrix}, \qquad v_{2}= \begin{pmatrix} 0 \\ 1 \end{pmatrix}, \qquad v_{3}= \begin{pmatrix} -r \\ s \end{pmatrix}, \qquad v_{4}= \begin{pmatrix} c \\ -\frac{cs}{r}-1 \end{pmatrix}, \] 
where $1 \leq s < r$ and $\gcd (r,s)=1$. Note that necessarily $v_{4}$ is an integer point, so $\frac{cs}{r} \in \Z$. Furthermore since $\gcd(s,r)=1$, it follows that $r \mid c$. Set $c= \tilde{c} r$ so that
\[  v_{4}= \begin{pmatrix} r \tilde{c} \\ -\tilde{c}s-1 \end{pmatrix}. \]
It follows that
\[ a_{1}=-\tilde{c}, \qquad \text{ and } \qquad a_{3} = \tilde{c}, \]
and the convexity condition $a_{i} > -2$, implies that $\lvert \tilde{c} \rvert \leq 1$. If $\tilde{c} = 0$, we have reduced to the previous case $c=0$. Therefore we can assume that $\tilde{c}=1$ and $v_{4} =  \begin{pmatrix} r \\ -s-1 \end{pmatrix}$, or $\tilde{c}=-1$ and $v_{4} =  \begin{pmatrix} -r \\ s-1 \end{pmatrix}$. This describes a third and fourth $k=4$ family, see Figure \ref{Fig2}.

\begin{figure}[H]
\centering
\begin{tabular}{| p{28pt} | p{133pt} | p{75pt} | p{172pt} |}
\hline
 & \begin{center} Polygon Model \end{center} & \begin{center} Conditions on variables \end{center} & \begin{center} Cyclic quotient singularities \end{center}\\ \hline \hline
 
 \begin{center} \vspace{35pt} Family 1 \end{center} & \vspace{0pt} \begin{tikzpicture}[scale=0.4, transform shape]
\begin{scope}
\clip (-6,-2.9) rectangle (6.3cm,2.8cm); 
\filldraw[fill=cyan, draw=blue] (5,-2) -- (0,1) -- (-5,2) -- (0,-1) -- (5,-2); 
\foreach \x in {-5,-4,...,5}{                           
    \foreach \y in {-7,-6,...,7}{                       
    \node[draw,shape = circle,inner sep=1pt,fill] at (\x,\y) {}; 
    }
}
 \node[draw,shape = circle,inner sep=4pt] at (0,0) {}; 
 \node at (-5,2.5) {$(-r,s)$};
 \node at (0,1.5) {$(0,1)$};
 \node at (5,-2.5) {$(r,-s)$};
  \node at (0,-1.5) {$(0,-1)$};
 \node at (2.5,-0.2) {$\sigma_{1}$};
 \node at (-2,1.7) {$\sigma_{2}$};
 \node at (-2.5,0) {$\sigma_{3}$};
  \node at (2.5,-1.8) {$\sigma_{4}$};
\end{scope}
\end{tikzpicture} & \vspace{15pt} \begin{center} \begin{tabular}{c} $1 \leq s < r$ \\ $\gcd(r,s)=1$ \end{tabular} \end{center} & \vspace{15pt} \begin{center} \begin{tabular}{c} $\sigma_{1} = \frac{1}{r}(1,s)$ \\ $\sigma_{2} = \frac{1}{r}(1,r- s)$ \\ $\sigma_{3} = \frac{1}{r}(1, s)$ \\ $\sigma_{4} = \frac{1}{r}(1, r -s)$.\end{tabular} \end{center} \\ \hline

 \begin{center} \vspace{35pt} Family 2 \end{center} & \vspace{0pt} \begin{tikzpicture}[scale=0.4, transform shape]
\begin{scope}
\clip (-6,-2.9) rectangle (6.3cm,3.8cm); 
\filldraw[fill=cyan, draw=blue] (5,-2) -- (0,1) -- (-5,3) -- (0,-1) -- (5,-2); 
\foreach \x in {-5,-4,...,5}{                           
    \foreach \y in {-7,-6,...,7}{                       
    \node[draw,shape = circle,inner sep=1pt,fill] at (\x,\y) {}; 
    }
}
 \node[draw,shape = circle,inner sep=4pt] at (0,0) {}; 
 \node at (-5,3.5) {$(-r,s+1)$};
 \node at (0,1.5) {$(0,1)$};
 \node at (5,-2.5) {$(r,-s)$};
  \node at (0,-1.5) {$(0,-1)$};
 \node at (2.5,-0.2) {$\sigma_{1}$};
 \node at (-2.5,2.2) {$\sigma_{2}$};
 \node at (-2.5,0.8) {$\sigma_{3}$};
  \node at (2.5,-1.8) {$\sigma_{4}$};
\end{scope}
\end{tikzpicture} & \vspace{15pt} \begin{center} \begin{tabular}{c} $1 \leq s < r$ \\ $\gcd(r,s)=1$ \\ $\gcd(r,s+1)=1$ \end{tabular} \end{center} & \vspace{15pt} \begin{center} \begin{tabular}{c} $\sigma_{1} = \frac{1}{r}(1,s)$ \\ $\sigma_{2} = \frac{1}{r}(1,r-1-s)$ \\ $\sigma_{3} = \frac{1}{r}(1, s+1)$ \\ $\sigma_{4} = \frac{1}{r}(1, r -s)$.\end{tabular} \end{center} \\ \hline

\begin{center} \vspace{35pt} Family 3 \end{center} & \vspace{0pt} \begin{tikzpicture}[scale=0.4, transform shape]
\begin{scope}
\clip (-6,-3.9) rectangle (6.3cm,2.8cm); 
\filldraw[fill=cyan, draw=blue] (5,-2) -- (0,1) -- (-5,2) -- (5,-3) -- (5,-2); 
\foreach \x in {-5,-4,...,5}{                           
    \foreach \y in {-7,-6,...,7}{                       
    \node[draw,shape = circle,inner sep=1pt,fill] at (\x,\y) {}; 
    }
}
 \node[draw,shape = circle,inner sep=4pt] at (0,0) {}; 
 \node at (-5,2.5) {$(-r,s)$};
 \node at (0,1.5) {$(0,1)$};
 \node at (5,-3.5) {$(r,-s-1)$};
  \node at (5,-1.5) {$(r,-s)$};
 \node at (2.5,0) {$\sigma_{1}$};
 \node at (-2,1.7) {$\sigma_{2}$};
 \node at (0.5,-1) {$\sigma_{3}$};
  \node at (5.5,-2.5) {$\sigma_{4}$};
\end{scope}
\end{tikzpicture}& \vspace{20pt} \begin{center} \begin{tabular}{c} $1 \leq s < r$ \\ $\gcd(r,s)=1$ \\ $\gcd(r,s+1)=1$ \end{tabular} \end{center} & \vspace{10pt} \begin{center} \begin{tabular}{c} $\sigma_{1} = \frac{1}{r}(1,s)$ \\ $\sigma_{2} = \frac{1}{r}(1,r- s)$ \\ $\sigma_{3} =\frac{1}{r}(r-s\left\lfloor \frac{r}{s+1} \right\rfloor, r - (s+1) \left\lfloor \frac{r}{s+1} \right\rfloor )$ \\ $\sigma_{4}$ is an unknown R-singularity.
\end{tabular} \end{center} \\ \hline

\begin{center} \vspace{35pt} Family 4 \end{center} & \vspace{0pt} \begin{tikzpicture}[scale=0.4, transform shape]
\begin{scope}
\clip (-6,-2.9) rectangle (6.3cm,2.8cm); 
\filldraw[fill=cyan, draw=blue] (5,-2) -- (0,1) -- (-5,2) -- (-5,1) -- (5,-2); 
\foreach \x in {-5,-4,...,5}{                           
    \foreach \y in {-7,-6,...,7}{                       
    \node[draw,shape = circle,inner sep=1pt,fill] at (\x,\y) {}; 
    }
}
 \node[draw,shape = circle,inner sep=4pt] at (0,0) {}; 
 \node at (-5,2.5) {$(-r,s)$};
 \node at (0,1.5) {$(0,1)$};
 \node at (-5,0.5) {$(-r,s-1)$};
  \node at (5,-2.5) {$(r,-s)$};
 \node at (2.5,0) {$\sigma_{1}$};
 \node at (-2,1.7) {$\sigma_{2}$};
 \node at (-5.5,1.5) {$\sigma_{3}$};
 \node at (0.5,-1) {$\sigma_{4}$};
\end{scope}
\end{tikzpicture}& \vspace{20pt} \begin{center} \begin{tabular}{c} $1 \leq s < r$ \\ $\gcd(r,s)=1$ \\ $\gcd(r,s-1)=1$ \end{tabular} \end{center} & \vspace{10pt} \begin{center} \begin{tabular}{c} $\sigma_{1} = \frac{1}{r}(1,s)$ \\ $\sigma_{2} = \frac{1}{r}(1,r- s)$ \\ $\sigma_{3}$ is an unknown R-singularity, \\ $\sigma_{4} =\frac{1}{r}(r-(s-1)\left\lfloor \frac{r}{s} \right\rfloor, r - s \left\lfloor \frac{r}{s} \right\rfloor )$
\end{tabular} \end{center} \\ \hline

\end{tabular}
\caption{Four families of Fano polygons with only determinant $r$ cones on 4 vertices.}\label{Fig2}
\end{figure}

Similarly to when $k=3$, not every Fano polygon here is $\ell$-reflexive. Indeed this is not even true of any individual family: consider $(r,s)=(15,2)$ for Family 1, $(r,s)=(15,7)$ for Family 2, $(r,s)=(15,13)$ for Family 3 and $(r,s)=(15,2)$ for Family 4. Each of these four explicit Fano polygons has four R-cones which are not all of equal height.

\subsection{Case $k=5$}

Assume 
\[ v_{1} = \begin{pmatrix} r \\ -s \end{pmatrix}, \qquad v_{2}= \begin{pmatrix} 0 \\ 1 \end{pmatrix}. \]
We then have 
\begin{align*}
\det (v_{2}, v_{3}) = r \qquad \implies& \qquad v_{3} = \begin{pmatrix} -r \\ \alpha \end{pmatrix}, \\
\det (v_{3}, v_{4}) = r \qquad \implies& \qquad v_{4} = \begin{pmatrix} r \beta \\ - \alpha \beta -1 \end{pmatrix}, \\
\det (v_{5}, v_{1}) = r \qquad \implies& \qquad v_{5} = \begin{pmatrix} r \gamma \\ - s \gamma -1 \end{pmatrix}.
\end{align*}
with the vertices further subject to the condition $\det (v_{4}, v_{5})=r$. 

First consider the case $\beta =0$. From the equation
\[ v_{3}+v_{5} + a_{4}v_{4} = \begin{pmatrix} -r \\ \alpha \end{pmatrix} + \begin{pmatrix} r \gamma \\ - s \gamma -1 \end{pmatrix} + a_{4} \begin{pmatrix} 0 \\ -1 \end{pmatrix} = \begin{pmatrix} 0 \\ 0 \end{pmatrix}, \]
obtain that necessarily $\gamma=1$. Therefore the vertices are given by
\[ v_{1} = \begin{pmatrix} r \\ -s \end{pmatrix}, \qquad v_{2}= \begin{pmatrix} 0 \\ 1 \end{pmatrix}, \qquad v_{3}= \begin{pmatrix} -r \\ \alpha \end{pmatrix}, \qquad v_{4}= \begin{pmatrix} 0 \\ -1 \end{pmatrix}, \qquad v_{5}= \begin{pmatrix} r \\ -s-1 \end{pmatrix}. \]
Calculating that $a_{2}=s-\alpha > -2$ and $a_{4}= \alpha -s-1 >-2$ implies that $\alpha \in \left\{ s, s+1 \right\}$. This gives us two families of Fano polygons, shown in Figure \ref{Fig3}, on five vertices such that each cone has determinant $r$.

The case $\gamma = 0$ is addressed similarly. The equation $v_{4} + v_{1} + a_{5}v_{5} =0$ implies that $\beta=-1$, and so the vertices of $P$ are given by
\[ v_{1} = \begin{pmatrix} r \\ -s \end{pmatrix}, \qquad v_{2}= \begin{pmatrix} 0 \\ 1 \end{pmatrix}, \qquad v_{3}= \begin{pmatrix} -r \\ \alpha \end{pmatrix}, \qquad v_{4}= \begin{pmatrix} -r \\ \alpha-1 \end{pmatrix}, \qquad v_{5}= \begin{pmatrix} 0 \\ -1 \end{pmatrix}. \]
The convexity condition for $a_{2}$ and $a_{5}$ imply that $\alpha \in \left\{ s, s+1 \right\}$ and we obtain two more families of Fano polygons, however these are both respectively isomorphic (with some relabelling) to one of the two families that occurred when $\beta=0$.

Having completed these two cases, instead assume both $\beta$ and $\gamma$ are non-zero. The $a_{i}$ for the associated $r$-modular sequence are:
\[ a_{1} = - \gamma, \qquad a_{2} = s - \alpha, \qquad a_{3} = \beta, \qquad a_{4} = \frac{1 - \gamma}{\beta}, \qquad \text{ and } \qquad a_{5} = \frac{-1 - \beta}{\gamma}.
\]
We split into four sub-cases:
\newline
\begin{tabular}{c p{80pt} c}
\begin{tabular}{l}
$(\text{i}) \text{  } \beta, \gamma > 0,$ \\
$(\text{ii}) \text{  } \beta, \gamma < 0,$
\end{tabular}
&
&
\begin{tabular}{l}
$(\text{iii}) \text{  } \beta > 0, \gamma < 0,$ \\
$(\text{iv}) \text{  } \beta < 0, \gamma > 0.$
\end{tabular}
\end{tabular}

In case (i), $a_{1} = - \gamma > -2$ implies that $0< \gamma < 2$ and so $\gamma =1$. It follows that
\[ a_{5} = -1-\beta > -2 \qquad \implies \qquad \beta < 1 \]
and since $\beta>0$, there are no possible integer solutions in case (i). Similarly in case (ii), $a_{3}>-2$ implies $-2 < \beta < 0$ and so $\beta=-1$. Manipulation of $a_{4}$, gives that $\gamma > -1$ and there are no integer solutions here either.

Suppose a polygon exists in case (iii). By Theorem \ref{2.3}, and noting that $a_{2} = s-\alpha > -2$:
\[ -3 = 12 - 3(5) = 12 - 3 \sum\limits_{i=1}^{5} \epsilon_{i} = \sum\limits_{i=1}^{5} a_{i} = - \gamma + (s-\alpha) + \beta + \frac{1-\gamma}{\beta} + \frac{-1-\beta}{\gamma} > 0 + 0 + (-2) + 0 + 0 = -2. \]
This is an obvious contradiction and no such polygon can exist.

Finally for case (iv), by the same method as in case (i) and case (ii), obtain bounds $0< \gamma < 2$ and $-2 < \beta < 0$, and hence $(\beta, \gamma) = (-1,1)$. Calculation of the winding number of Theorem \ref{2.3} forces $\alpha = s+1$, and we obtain a final family in the case $k=5$ with vertices:
\[ v_{1} = \begin{pmatrix} r \\ -s \end{pmatrix}, \qquad v_{2}= \begin{pmatrix} 0 \\ 1 \end{pmatrix}, \qquad v_{3}= \begin{pmatrix} -r \\ s+1 \end{pmatrix}, \qquad v_{4}= \begin{pmatrix} -r \\ s \end{pmatrix}, \qquad v_{5}= \begin{pmatrix} r \\ -s-1 \end{pmatrix}. \]

\begin{figure}[H]
\centering
\begin{tabular}{| p{28pt} | p{133pt} | p{75pt} | p{172pt} |}
\hline
 & \begin{center} Polygon Model \end{center} & \begin{center} Conditions on variables \end{center} & \begin{center} Cyclic quotient singularities \end{center}\\ \hline \hline
 
 \begin{center} \vspace{25pt} Family 1 \end{center} & \vspace{0pt}  \begin{tikzpicture}[scale=0.4, transform shape]
\begin{scope}
\clip (-6,-2.9) rectangle (6.3cm,3.8cm); 
\filldraw[fill=cyan, draw=blue] (5,-2) -- (0,1) -- (-5,3) -- (-5,2) -- (0,-1) -- (5,-2); 
\foreach \x in {-5,-4,...,5}{                           
    \foreach \y in {-7,-6,...,7}{                       
    \node[draw,shape = circle,inner sep=1pt,fill] at (\x,\y) {}; 
    }
}
 \node[draw,shape = circle,inner sep=4pt] at (0,0) {}; 
 \node at (-5,3.5) {$(-r,s+1)$};
 \node at (-5.1,1.3) {$(-r,s)$};
 \node at (0,1.5) {$(0,1)$};
  \node at (5,-2.5) {$(r,-s)$};
   \node at (0,-1.5) {$(0,-1)$};
 \node at (2.5,0) {$\sigma_{1}$};
 \node at (-2,2.3) {$\sigma_{2}$};
  \node at (-5.5,2.5) {$\sigma_{3}$};
 \node at (-2.5,0) {$\sigma_{4}$};
 \node at (2,-1.7) {$\sigma_{5}$};
\end{scope}
\end{tikzpicture}  & \vspace{15pt} \begin{center} \begin{tabular}{c} $1 \leq s < r$, \\ $\gcd(r,s)=1$, \\ $\gcd(r,s+1)=1$. \end{tabular} \end{center} & \vspace{5pt} \begin{center} \begin{tabular}{c} $\sigma_{1} = \frac{1}{r}(1,s)$, \\ $\sigma_{2} = \frac{1}{r}(1,r- 1-s)$, \\ $\sigma_{3}$ is an unknown R-singularity, \\ $\sigma_{4} = \frac{1}{r}(1,s)$, \\ $\sigma_{5} = \frac{1}{r}(1,r- s)$.\end{tabular} \end{center} \\ \hline

\begin{center} \vspace{35pt} Family 2 \end{center} & \vspace{0pt} \begin{tikzpicture}[scale=0.4, transform shape]
\begin{scope}
\clip (-6,-2.9) rectangle (6.3cm,2.8cm); 
\filldraw[fill=cyan, draw=blue] (5,-2) -- (0,1) -- (-5,2) -- (-5,1) -- (0,-1) -- (5,-2); 
\foreach \x in {-5,-4,...,5}{                           
    \foreach \y in {-7,-6,...,7}{                       
    \node[draw,shape = circle,inner sep=1pt,fill] at (\x,\y) {}; 
    }
}
 \node[draw,shape = circle,inner sep=4pt] at (0,0) {}; 
 \node at (-5,2.5) {$(-r,s)$};
 \node at (-5.1,0.3) {$(-r,s-1)$};
 \node at (0,1.5) {$(0,1)$};
  \node at (5,-2.5) {$(r,-s)$};
   \node at (0,-1.5) {$(0,-1)$};
 \node at (2.5,0) {$\sigma_{1}$};
 \node at (-2,1.7) {$\sigma_{2}$};
  \node at (-5.5,1.5) {$\sigma_{3}$};
 \node at (-2.5,-0.5) {$\sigma_{4}$};
 \node at (2,-1.7) {$\sigma_{5}$};

\end{scope}
\end{tikzpicture} &  \vspace{10pt} \begin{center} \begin{tabular}{c} $1 \leq s < r$, \\ $\gcd(r,s)=1$, \\ $\gcd(r,s-1)=1$. \end{tabular} \end{center}  & \vspace{5pt} \begin{center} \begin{tabular}{c} $\sigma_{1} = \frac{1}{r}(1,s)$, \\ $\sigma_{2} = \frac{1}{r}(1,r- s)$, \\ $\sigma_{3}$ is an unknown R-singularity, \\ $\sigma_{4} = \frac{1}{r}(1,s -1)$, \\ $\sigma_{5} = \frac{1}{r}(1,r- s)$. \end{tabular} \end{center} \\ \hline

\begin{center} \vspace{30pt} Family 3 \end{center} & \vspace{0pt} \begin{tikzpicture}[scale=0.4, transform shape]
\begin{scope}
\clip (-6,-3.9) rectangle (6.3cm,3.8cm); 
\filldraw[fill=cyan, draw=blue] (5,-2) -- (0,1) -- (-5,3) -- (-5,2) -- (5,-3) -- (5,-2); 
\foreach \x in {-5,-4,...,5}{                           
    \foreach \y in {-7,-6,...,7}{                       
    \node[draw,shape = circle,inner sep=1pt,fill] at (\x,\y) {}; 
    }
}
 \node[draw,shape = circle,inner sep=4pt] at (0,0) {}; 
  \node at (-5,3.5) {$(-r,s+1)$};
 \node at (-5,1.5) {$(-r,s)$};
 \node at (0,1.5) {$(0,1)$};
 \node at (5,-3.5) {$(r,-s-1)$};
  \node at (5,-1.5) {$(r,-s)$};
 \node at (2.5,0) {$\sigma_{1}$};
 \node at (-2,2.3) {$\sigma_{2}$};
 \node at (-5.5,2.5) {$\sigma_{3}$};
 \node at (0,-1.5) {$\sigma_{4}$};
  \node at (5.5,-2.5) {$\sigma_{5}$};
\end{scope}
\end{tikzpicture} & \vspace{20pt} \begin{center} \begin{tabular}{c} $1 \leq s < r$, \\ $\gcd(r,s)=1$, \\ $\gcd(r,s+1)=1$. \end{tabular} \end{center}  & \vspace{10pt} \begin{center} \begin{tabular}{c} $\sigma_{1} = \frac{1}{r}(1,s)$, \\ $\sigma_{2} = \frac{1}{r}(1,r- 1-s)$, \\ $\sigma_{3}$ is an unknown R-singularity, \\ $\sigma_{4} =\frac{1}{r}(r-s \left\lfloor \frac{r}{s+1} \right\rfloor, r - (s+1) \left\lfloor \frac{r}{s+1} \right\rfloor )$, \\
$\sigma_{5}$ is an unknown R-singularity. \end{tabular} \end{center} \\ \hline

\end{tabular}
\caption{Three families of Fano polygons with only determinant $r$ cones on 5 vertices.}\label{Fig3}
\end{figure}

Again each family here contains a Fano polygon that has five R-cones and is not $\ell$-reflexive: consider $(r,s)=(15,7)$ for Family 1, $(r,s)=(15,2)$ for Family 2 and $(r,s)=(15,13)$ for Family 3.
\subsection{Case $k=6$}

Consider a Fano polygon on six vertices all of whose cones have determinant $r$. By studying the winding number equation given in Theorem \ref{2.3} obtain:
\[ \sum\limits_{i=1}^{6} a_{i} + 3 \sum\limits_{i=1}^{6} (1) = 12, \qquad \implies \qquad \sum\limits_{i=1}^{6} a_{i} = -6. \]
This has a unique solution given by $a_{i}=-1, \forall i$, since by convexity $a_{i}>-2$. Hence having fixed $\sigma_{1}$, that is, $v_{1} = \begin{pmatrix} r \\ -s \end{pmatrix}$ where $1 \leq s < r$ with $\gcd (r,s)=1$, and $v_{2}= \begin{pmatrix} 0 \\ 1 \end{pmatrix}$, all other vertices are determined by the identity:
\[ \begin{pmatrix} v_{i} & v_{i+1} \end{pmatrix} = \begin{pmatrix} v_{i-1} & v_{i} \end{pmatrix}  \begin{pmatrix} 0 & -1 \\ 1 & 1 \end{pmatrix}. \]
Namely
\[ v_{3} = \begin{pmatrix} -r \\ s+1 \end{pmatrix}, \qquad v_{4} = \begin{pmatrix} -r \\ s \end{pmatrix}, \qquad v_{5} = \begin{pmatrix} 0 \\ -1 \end{pmatrix}, \qquad v_{6} = \begin{pmatrix} r \\ -s-1 \end{pmatrix}. \]
Convexity is already satisfied by construction and therefore the only required conditions are that $\gcd(r,s)= \gcd(r,s+1) =1$. This unique model for the $k=6$ case is as shown:

\begin{figure}[H]
\centering
\begin{tabular}{| p{28pt} | p{133pt} | p{80pt} | p{167pt} |}
\hline
 & \begin{center} Polygon Model \end{center} & \begin{center} Conditions on variables \end{center} & \begin{center} Cyclic quotient singularities \end{center}\\ \hline \hline
 
 \begin{center} \vspace{25pt} Family 1 \end{center} & \vspace{0pt} \begin{tikzpicture}[scale=0.4, transform shape]
\begin{scope}
\clip (-6,-3.9) rectangle (6.3cm,3.8cm); 
\filldraw[fill=cyan, draw=blue] (5,-2) -- (0,1) -- (-5,3) -- (-5,2) -- (0,-1) -- (5,-3) -- (5,-2); 
\foreach \x in {-5,-4,...,5}{                           
    \foreach \y in {-7,-6,...,7}{                       
    \node[draw,shape = circle,inner sep=1pt,fill] at (\x,\y) {}; 
    }
}
 \node[draw,shape = circle,inner sep=4pt] at (0,0) {}; 
  \node at (-5,3.5) {$(-r,s+1)$};
 \node at (-5,1.5) {$(-r,s)$};
 \node at (0,1.5) {$(0,1)$};
 \node at (5,-3.5) {$(r,-s-1)$};
  \node at (5,-1.5) {$(r,-s)$};
   \node at (0,-1.5) {$(0,-1)$};
 \node at (2.5,0) {$\sigma_{1}$};
 \node at (-2,2.3) {$\sigma_{2}$};
 \node at (-5.5,2.5) {$\sigma_{3}$};
 \node at (-2.5,0) {$\sigma_{4}$};
 \node at (2,-2.3) {$\sigma_{5}$};
  \node at (5.5,-2.5) {$\sigma_{6}$};
\end{scope}
\end{tikzpicture}
& \vspace{15pt} \begin{center} \begin{tabular}{c} $1 \leq s < r$, \\ $\gcd(r,s)=1$, \\ $\gcd(r,s+1)=1$. \end{tabular} \end{center} & \vspace{8pt} \begin{center} \begin{tabular}{c} $\sigma_{1} = \frac{1}{r}(1,s)$, \\ $\sigma_{2} = \frac{1}{r}(1,r- 1-s)$, \\ $\sigma_{3}$ is an unknown R-singularity, \\  $\sigma_{4} = \frac{1}{r}(1,s)$, \\ $\sigma_{5} = \frac{1}{r}(1,r- 1-s)$, \\ $\sigma_{6}$ is an unknown R-singularity. \end{tabular} \end{center} \\ \hline

\end{tabular}
\caption{Unique familiy of Fano polygons with only determinant $r$ cones on 6 vertices.}\label{Fig4}
\end{figure}

This $k=6$ family is the only family appearing in the paper for which every entry is $\ell$-reflexive for some $\ell \in \Z_{>0}$. More specifically every Fano polygon here is $r$-reflexive. This is guaranteed by the conditions $\gcd(r,s)=\gcd(r,s+1)=1$. 

\section{Fano Polygons with Singularity Content $\left( 0, \left\{ k \times \frac{1}{r}(1,s) \right\} \right)$}\label{Section4}

The aim of this section is to provide a classification of Fano polygons with singularity content of the form $\left( 0, \left\{ k \times \frac{1}{r}(1,s) \right\} \right)$. Note any polygon appearing here, arises as an $l$-reflexive polygon in \cite{ReflexivePolytopesofHigherIndexandtheNumber12}. Since the cone $C_{\frac{1}{r}(1,s)}$ has determinant $r$, every object in this classification will appear as one of the models derived in Section \ref{3}. We analyse each of the derived eight families of Fano polygons with only determinant $r$ cones, and deduce conditions on the parameters $r,s$ under which every cone of a family member corresponds to the same $\frac{1}{r}(1,s)$ cyclic quotient R-singularity. Note immediately that necessarily $r>2$ since $\frac{1}{1}(1,1)$ and $\frac{1}{2}(1,1)$ are T-singularities.

\subsection{Useful Lemmas}

We provide three results that will be used repeatedly. These provide conditions on the variables $r$ and $s$ as to when certain R-singularities will be isomorphormic.

\begin{lem}\label{4.1}
The cones representing the R-singularities $\frac{1}{r}(1,s)$ and $\frac{1}{r}(1,r -s \pm 1)$ are isomorphic if and only if either $s=\frac{r \pm 1}{2}$, or $p \equiv 1 \modb{6}$ for all primes $p \mid r$ and $s^{2} \mp s + 1 \equiv 0 \modb{r}$.
\end{lem}

\begin{proof}
There are two conditions under which these cones are isomorphic, namely:
\begin{enumerate}[label=(\roman*)]
\item $s=r-s \pm 1$, 
\item $s(r-s \pm 1) \equiv 1 \modb{r}$.
\end{enumerate}
In case (i) it follows trivially that $s=\frac{r \pm 1}{2}$. 

Now consider case (ii):
\[ s(r-s \pm 1) \equiv 1 \modb{r} \qquad \iff \qquad s^{2} \mp s+ 1 \equiv 0 \modb{r}. \]
Techniques in number theory tell us that a quadratic congruence has a solution if and only if the square root of the discriminant of the quadratic exists in the finite field. In this case the square root of the discriminant is given by $\sqrt{(\pm 1)^{2}-4(1)(1)}= \sqrt{-3}$. It is a well known result concerning Legendre symbols that for $p$ be an odd prime, then 
\[ \left(\frac{-3}{p}\right) = \begin{cases} 1, \qquad \text{ if } p \equiv 1 \modb{3}, \\ -1, \qquad \text{if } p \equiv -1 \modb{3}. \end{cases} \]
Combining this identity with the Chinese remainder theorem and the fact that we know $r$ to be odd means $s$ exists if and only $p \equiv 1 \modb{6}$ for all primes $p \mid r$. It is not clear however how to express $s$ in terms of $r$ with this information.
\end{proof}

\begin{lem}\label{4.3}
The cones representing cyclic quotient R-singularities $\frac{1}{r}(1,s)$ and $\frac{1}{r}(1,r-s)$ are isomorphic if and only if $p \equiv 1 \modb{4}$ for all primes $p \mid r$, and $s^{2} \equiv -1 \modb{r}$.
\end{lem}

\begin{proof}
Note that $\frac{1}{r}(1,s) \cong \frac{1}{r}(1,r-s)$ if and only if either
\begin{enumerate}[label=(\roman*)]
\item $s=r-s$;
\item $s(r-s) \equiv 1 \modb{r}$.
\end{enumerate}

If (i) holds it follows that $s \mid r$, and since $\gcd(r,s)=1$ that $s=1$ and $r=2$ which is not of interest. Therefore the cones are isomorphic R-cones if and only if $s(r-s) \equiv 1 \modb{r}$, which is equivalent to the quadratic congruence:
\[ s^{2} \equiv -1 \modb{r}. \]
Similarly to the proof of Lemma \ref{4.1}, we want to understand when the discriminant $-1$ is a square in the finite ring $\Z / r \Z$. It is well known that for an odd prime $p$:
\[ \left(\frac{-1}{p}\right) = \begin{cases} 1, \qquad \text{ if } p \equiv 1 \modb{x4}, \\ -1, \qquad \text{if } p \equiv 3 \modb{4}. \end{cases} \]
By the Chinese Remainder Theorem such a solution for $s$ exists if $p \equiv 1 \modb{4}$ holds for all primes dividing $r$.
\end{proof}

\begin{lem}\label{4.4}
Consider three cones representing R-singularities $\frac{1}{r}(1,s), \frac{1}{r}(1,r-s)$ and $\frac{1}{r}(1,r-s-1)$. Then all three cones are isomorphic if and only if $(r,s)=(5,2)$.
\end{lem}

\begin{proof}
The $\frac{1}{r}(1,s)$ and $\frac{1}{r}(1,r-1-s)$ cones imply that either $s=\frac{r-1}{2}$ or $s^{2}+s+1 \equiv 0 \modb{r}$ by Lemma \ref{4.1}. Similarly the cones $\frac{1}{r}(1,s)$ and $\frac{1}{r}(1, r- s)$ imply by Lemma \ref{4.4} that $s^{2} \equiv -1 \modb{r}$. It follows that 
 \[ 0 \equiv s^{2}+s+1 \equiv s \modb{r}, \]
 which is not possible since $\gcd (r,s)=1$ and $r>2$. Hence the only remaining interesting possibility is $s=\frac{r-1}{2}$, which means
 \begin{align*}
 \left( \frac{r-1}{2} \right)^{2} \equiv& -1 \modb{r}, \\
 (-1)^{2} \equiv& -4 \modb{r}, \\
 0 \equiv& 5 \modb{r}.
 \end{align*}
 Therefore $(r,s)=(5,2)$
\end{proof}

\subsection{Case $k=3$}

There is a unique family of Fano polygons on three vertices all of whose cones have determinant $r$, shown in Figure \ref{Fig1}. Consider
\[ \sigma_{1} = \frac{1}{r}(1,s), \qquad \text{ and } \qquad \sigma_{2} = \frac{1}{r}(1,r-s+1). \]
By Lemma \ref{4.1} there are two possibilies: $s=\frac{r+1}{2}$ or $s^{2} - s + 1 \equiv 0 \modb{r}$. Suppose $s=\frac{r+1}{2}$, and look at $\sigma_{3}$ under this assumption:

\begin{align*}
\sigma_{3} &= \frac{1}{r} \left( r + \left( 1 - \frac{r+1}{2} \right) \left\lfloor \frac{r}{\frac{r+1}{2}} \right\rfloor, r - \frac{r+1}{2} \left\lfloor \frac{r}{\frac{r+1}{2}} \right\rfloor \right) \\
&= \frac{1}{r} \left( \frac{r+1}{2}, \frac{r-1}{2} \right) \\
&= \frac{1}{r} \left(1, \frac{r+1}{2} \left(  \frac{r-1}{2} \right)^{-1} \right).
\end{align*}

So in this case $\sigma_{3} \cong \sigma_{1}$ if and only if either 

\begin{itemize}
\item $\frac{r+1}{2}  \left(  \frac{r-1}{2} \right)^{-1} \equiv \frac{r+1}{2} \modb{r}$,
\item $\frac{r+1}{2}  \left(  \frac{r-1}{2} \right)^{-1} \frac{r+1}{2} \equiv 1 \modb{r}$. 
\end{itemize}

Suppose

\begin{align*}
\frac{r+1}{2}  \left(  \frac{r-1}{2} \right)^{-1} &\equiv \frac{r+1}{2} \modb{r}, \\
\frac{r-1}{2} &\equiv 1 \modb{r}, \\
r-1 &\equiv 2 \modb{r}, \\
3 &\equiv 0 \modb{r},
\end{align*}

which implies $r=3$, giving $s=2$. However the cyclic quotient singularity $\frac{1}{3}(1,2)$ is a T-singularity and so does not contribute a polygon to our classification. Alternatively suppose

\begin{align*}
\frac{r+1}{2}  \left(  \frac{r-1}{2} \right)^{-1} \frac{r+1}{2} &\equiv 1 \modb{r}, \\
(r+1)(r+1) &\equiv 2(r-1) \modb{r}, \\
1 &\equiv -2 \modb{r}, \\
3 &\equiv 0 \modb{r},
\end{align*}

which provides the same restrictions on $r$.

Alternatively consider $\sigma_{3}$ in the case $s^{2} - s + 1 \equiv 0 \modb{r}$:

\begin{align*}
\sigma_{3} &= \frac{1}{r} \left( r - \left( s -1 \right) \left\lfloor \frac{r}{s} \right\rfloor, r - s \left\lfloor \frac{r}{s} \right\rfloor \right), \\
&\cong \frac{1}{r} \left( 1, \left( r - \left( s -1\right) \left\lfloor \frac{r}{s} \right\rfloor \right) \left( r - s \left\lfloor \frac{r}{s} \right\rfloor \right)^{-1} \right), \\
&= \frac{1}{r} \left( 1, * \right).
\end{align*}

As previously $\sigma_{3} \cong \sigma_{1}$ if and only if either
\begin{itemize}
\item $* \equiv s \modb{r}$, 
\item $* \equiv s^{-1} \modb{r}$.
\end{itemize}

We claim the first of these two alternatives always holds:

\begin{align*}
* &\equiv s \modb{r}, \\
r - \left( s - 1 \right) \left\lfloor \frac{r}{s} \right\rfloor &\equiv s \left( r - s \left\lfloor \frac{r}{s} \right\rfloor \right) \modb{r}, \\
-(s-1) \left\lfloor \frac{r}{s} \right\rfloor &\equiv -s^{2} \left\lfloor \frac{r}{s} \right\rfloor \modb{r}, \\
(s^{2}-s+1) \left\lfloor \frac{r}{s} \right\rfloor &\equiv 0 \modb{r}.
\end{align*}

which is true by our condition on $s$, and so $\sigma_{1} \cong \sigma_{2} \cong \sigma_{3}$ for such a choice of $r$ and $s$.

\begin{prop}\label{4.4}
There exists a Fano polygon $P$ such that 
\[  \text{SC}(P) = \left( 0, \left\{ 3 \times \frac{1}{r}(1,s) \right\} \right), \]
if and only if $p \equiv 1 \modb{6}$ for all primes $p \mid r$, and $s^{2}-s+1 \equiv 0 \modb{r}$.
\end{prop}

\subsection{Case $k=4$}

There are four families to analyse for $k=4$. We check all four in turn and attempt to force that all the cones be isomorphic as $\frac{1}{r}(1,s)$ R-cones. The first family of Fano polygons with $k=4$ shown in Figure \ref{Fig2}, has 4 cones representing two different singularities in total. By Lemma \ref{4.3} applied to this family there is a collection of Fano polygons all of whose cones represent $\frac{1}{r}(1,s)$ singularities where $s^{2} \equiv -1 \modb{r}$ and  $p \equiv 1 \modb{4}$ for all primes $p \mid r$.

Lemma \ref{4.4} shows that the only possibility for the second family is when $(r,s)=(5,2)$. It is routine to check that for $(r,s)=(5,2)$ the only other remaining cone $\sigma_{3}$ is also isomorphic to the others, and so a single polygon arises from this family. This Fano polygon has singularity content $\left( 0, \left\{ 4 \times \frac{1}{5}(1,2) \right\} \right)$, the same singularity content as the polygons arising when $(r,s)=(5,2)$ in family 1, however these are two non-isomorphic Fano polygons. 

We now look at the third family of polygons shown in Figure \ref{Fig2}. As in the previous family, the cones $\sigma_{1}$ and $\sigma_{2}$ representing the singularities $\frac{1}{r}(1,s)$ and $\frac{1}{r}(1,r-s)$ imply by Lemma \ref{4.3}, that $s^{2} \equiv -1 \modb{r}$ and all primes dividing $r$ satisfy $p \equiv 1 \modb{4}$. It remains to check $\sigma_{3}$ and $\sigma_{4}$. The cone $\sigma_{3}$ can be written as:
\[ \sigma_{3} = \frac{1}{r} \left( 1, \left( r - (s+1) \left\lfloor \frac{r}{s+1} \right\rfloor \right) \left( r - s \left\lfloor \frac{r}{s+1} \right\rfloor \right)^{-1} \right) = \frac{1}{r} \left( 1,* \right) \]
 and is isomorphic to $\sigma_{1}$ if and only if either:
 \begin{enumerate}[label=(\roman*)]
 \item $* \equiv s \modb{r}$;
 \item $* \equiv s^{-1} \equiv -s \modb{r}$.
 \end{enumerate}
 
Note in both these cases that $\gcd \left(r, \left\lfloor \frac{r}{s+1} \right\rfloor \right)=1$, since the singularity $\sigma_{3}$ is well-defined, and so $\left\lfloor \frac{r}{s+1} \right\rfloor$ is invertible modulo $r$.

Consider case (ii) first:
\begin{align*}
\left( r - (s+1) \left\lfloor \frac{r}{s+1} \right\rfloor \right) \left( r - s \left\lfloor \frac{r}{s+1} \right\rfloor \right) ^{-1} \equiv& -s \modb{r}, \\
-(s+1) \left\lfloor \frac{r}{s+1} \right\rfloor \equiv& s^{2}  \left\lfloor \frac{r}{s+1} \right\rfloor \modb{r}, \\
-s-1 \equiv& -1 \modb{r}, \\
s \equiv& 0 \modb{r}.
\end{align*}

This is not possible since $\gcd(s,r)=1$, and $r>2$.

In case (i), we have
\begin{align*}
\left( r - (s+1) \left\lfloor \frac{r}{s+1} \right\rfloor \right) \left( r - s \left\lfloor \frac{r}{s+1} \right\rfloor \right)^{-1}  \equiv& s \modb{r}, \\
-(s+1) \left\lfloor \frac{r}{s+1} \right\rfloor \equiv& - s^{2}  \left\lfloor \frac{r}{s+1} \right\rfloor \modb{r}, \\
-s-1 \equiv& 1 \modb{r}, \\
s \equiv& -2 \modb{r}.
\end{align*}

Therefore $s=r-2$. We require $s^{2} \equiv (r-2)^2 \equiv (-2)^{2} \equiv -1 \modb{r}$ which implies $r=5$ and $s=3$. For these values, $\sigma_{4}$ has vertices $(5,-3)$ and $(5,-4)$ which we know to describe a $\frac{1}{5}(1,2)$ cyclic quotient singularity. Therefore $(r,s)=(5,3)$ describes a suitable Fano polygon, however it is easy to check that this polygon is isomorphic to the $k=4$ family 2 polygon for $(r,s) = (5,2)$.

The fourth family follows very similarly to the third. The cones $\sigma_{1}$ and $\sigma_{2}$ being isomorphic is equivalent to $s^{2} \equiv -1 \modb{r}$ by Lemma \ref{4.3}. Then $\sigma_{4} = \frac{1}{r}(r-(s-1)\left\lfloor \frac{r}{s} \right\rfloor, r - s \left\lfloor \frac{r}{s} \right\rfloor )$ is isomorphic to $\sigma_{1}$ and $\sigma_{2}$ if and only if either
 \begin{enumerate}[label=(\roman*)]
 \item $r - s \left\lfloor \frac{r}{s} \right\rfloor \equiv s \left( r - (s-1) \left\lfloor \frac{r}{s} \right\rfloor \right) \modb{r}$;
 \item $r - s \left\lfloor \frac{r}{s} \right\rfloor \equiv s^{-1} \left( r - (s-1) \left\lfloor \frac{r}{s} \right\rfloor \right) \modb{r}$.
 \end{enumerate} 
Consider case (ii):
\begin{align*}
r - s \left\lfloor \frac{r}{s} \right\rfloor \equiv& s^{-1} \left( r - (s-1) \left\lfloor \frac{r}{s} \right\rfloor \right) \modb{r}, \\
- s \left\lfloor \frac{r}{s} \right\rfloor \equiv& s (s-1) \left\lfloor \frac{r}{s} \right\rfloor  \modb{r}, \\
-s \equiv& s^{2} -s \modb{r}, \\
0 \equiv& -1 \modb{r},
\end{align*}
which is impossible since $r>2$. Finally for case (i):
\begin{align*}
r - s \left\lfloor \frac{r}{s} \right\rfloor \equiv& s \left( r - (s-1) \left\lfloor \frac{r}{s} \right\rfloor \right) \modb{r}, \\
- s \left\lfloor \frac{r}{s} \right\rfloor \equiv& -s (s-1) \left\lfloor \frac{r}{s} \right\rfloor \modb{r},  \\
- s \equiv& -s^{2}+s \modb{r}, \\
s^{2} \equiv& 2s \modb{r}, \\
s \equiv& 2 \modb{r}, \\
-1 \equiv s^{2} \equiv& 4 \modb{r}, \\
0 \equiv& 5 \modb{r}.
\end{align*}
Therefore either $(r,s)=(5,2)$ for which the polygon is isomorphic to when $(r,s)=(5,2)$ in family 2 of Figure \ref{Fig2}, or $(r,s)=(5,3)$ for which $\sigma_{4}$ is not then isomorphic to the other cones of the polygon.

We have deduced the following proposition:

\begin{prop}\label{4.5}
There exists a Fano polygon $P$ such that 
\[  \text{SC}(P) = \left( 0, \left\{ 4 \times \frac{1}{r}(1,s) \right\} \right), \]
if and only if $p \equiv 1 \modb{4}$ for all primes $p \mid r$, and $s^{2} \equiv -1 \modb{r}$. Furthermore this polygon is unique up to isomorphism, except in the case $(r,s) = (5,2)$ where there are two non-isomorphic polygons with the described singularity content.
\end{prop}

\subsection{Case $k=5$}

Consider the first  family of Fano polygons with cones of determinant $r$ shown in Figure \ref{Fig3}. By Lemma \ref{4.4}, the only possibility is $(r,s)=(5,2)$. The ray generators of $\sigma_{5}$ are now $(5,-3)$ and $(5,-2)$ and describe a $\frac{1}{5}(1,1)$ cone which is not isomorphic to the other $\sigma_{i}$. No polygons of interest arise in the first $k=5$ family.
 
Consider the second $k=5$ family shown in Figure \ref{Fig3} whose analysis follows very similarly to that of the first family. Indeed the same application of Lemma \ref{4.1} implies $s^{2} \equiv -1 \modb{r}$. In order for $\sigma_{1} = \frac{1}{r}(1,s)$ and $\sigma_{4} = \frac{1}{r}(1,s-1)$ to be isomorphic, we require $s(s-1) \equiv 1 \modb{r}$ since $s \equiv s-1 \modb{r}$ is clearly not giving rise to a suitable polygon. So
\begin{align*}
s(s-1) \equiv& 1 \modb{r}, \\
s^{2}-s \equiv& 1 \modb{r}, \\
-1-s \equiv& 1 \modb{r}, \\
s \equiv& -2 \modb{r}.
\end{align*}
We have already seen that $s\equiv -2 \modb{r}$ combined with $s^{2} \equiv -1 \modb{r}$ leads to $(r,s)=(5,3)$. It follows though that $\sigma_{3}$ is not isomorphic to the other cones when these values are taken. No suitable polygons arise here either.

There is a final remaining family for $k=5$ shown in Figure \ref{Fig3}. In order for $\sigma_{1} = \frac{1}{r}(1,s)$ and $\sigma_{2} = \frac{1}{r}(1,r-s-1)$ to be isomorphic it is required by Lemma \ref{4.1} that either
\begin{enumerate}[label=(\roman*)]
\item $s=\frac{r-1}{2}$, 
\item $s^{2}+s+1 \equiv 0 \modb{r}$.
\end{enumerate}

Studying $\sigma_{4}$ in case (i):
\begin{align*}
\sigma_{4} =& \frac{1}{r} \left( r - \left( \frac{r-1}{2} \right) \left\lfloor \frac{r}{\frac{r+1}{2}} \right\rfloor, r - \left( \frac{r+1}{2} \right) \left\lfloor \frac{r}{\frac{r+1}{2}} \right\rfloor \right) \\
=& \frac{1}{r} \left( \frac{r+1}{2}, \frac{r-1}{2} \right) \\
=& \frac{1}{r} \left( 1, \left( \frac{r+1}{2} \right)^{-1} \left( \frac{r-1}{2} \right) \right).
\end{align*}

So for $\sigma_{4}$ to be isomorphic to $\sigma_{1}$ and $\sigma_{2}$ implies either 
\[ \left( \frac{r+1}{2} \right)^{-1} \left( \frac{r-1}{2} \right) = \frac{r-1}{2}, \qquad \text{ or } \qquad \left( \frac{r+1}{2} \right)^{-1} \left( \frac{r-1}{2} \right)^{2} \equiv 1 \modb{r},\] both of which give $r=1$ meaning the cones represent T-singularities and so are not of interest.

Alternatively in case (ii) where $s^{2}+s+1 \equiv 0 \modb{r}$, we can write
\[ \sigma_{4} =  \frac{1}{r} \left( 1, \left( r - \left( s+1 \right) \left\lfloor \frac{r}{s+1} \right\rfloor \right) \left( r - s \left\lfloor \frac{r}{s+1} \right\rfloor \right)^{-1} \right).  \]
Suppose 
\begin{align*}
\left( r - \left( s+1 \right) \left\lfloor \frac{r}{s+1} \right\rfloor \right) \left( r - s \left\lfloor \frac{r}{s+1} \right\rfloor \right)^{-1} \equiv& s \modb{r}, \\
-(s+1) \equiv& - s^{2} \modb{r}, \\
s^{2} - s -1 \equiv& 0 \modb{r}.
\end{align*}
However this equation is unsolvable alongside $s^{2} + s +1 \equiv 0 \modb{r}$ since together they imply $2s^{2} \equiv 0 \modb{r}$, and we know $r>2$ and $\gcd(r,s)=1$. Instead if 
\begin{align*}
\left( r - \left( s+1 \right) \left\lfloor \frac{r}{s+1} \right\rfloor \right) \left( r - s \left\lfloor \frac{r}{s+1} \right\rfloor \right)^{-1} \equiv& s^{-1} \equiv r-s-1 \modb{r}, \\
-(s+1)  \left\lfloor \frac{r}{s+1} \right\rfloor \equiv& (-s-1) \left( -s \left\lfloor \frac{r}{s+1} \right\rfloor \right) \modb{r}, \\
-s -1 \equiv& s^{2} +s \modb{r}, \\
s^{2} +2s +1 \equiv& 0 \modb{r},  \\
s \equiv& 0 \modb{r},
\end{align*}
which we know to be impossible. There are no cases of a Fano polygon all of whose cones represent the same R-singularity in this final family.

\begin{prop}\label{4.6}
There does not exist a Fano polygon $P$ such that 
\[  \text{SC}(P) = \left( 0, \left\{ 5 \times \frac{1}{r}(1,s) \right\} \right). \]
\end{prop}

\subsection{Case $k=6$}

There is a unique family of Fano polygons with $k=6$ shown in Figure \ref{Fig4}, each with 3 different types of cones since $\sigma_{1} \cong \sigma_{4}$, $\sigma_{2} \cong \sigma_{5}$ and $\sigma_{3} \cong \sigma_{6}$. By Lemma \ref{4.1}, $\sigma_{1} \cong \sigma_{2}$ if and only if either
\begin{enumerate}[label=(\roman*)]
\item $s=\frac{r-1}{2}$, 
\item $s^{2}+s+1 \equiv 0 \modb{r}$.
\end{enumerate}

In case (i), we study $\sigma_{3}$ and when it is isomorphic to a $\frac{1}{r} \left( 1,\frac{r-1}{2} \right)$ cone. Namely $\sigma_{3}$ has ray generators $\left( -r, \frac{r+1}{2} \right)$ and $\left( -r, \frac{r-1}{2} \right)$, and hence is a $\frac{1}{r}(1,1)$ cone. Therefore $\sigma_{3}$ is a $\frac{1}{r}(1,\frac{r-1}{2})$ cone if and only if $r = 3$. Hence $(r,s)=(3,1)$ gives rise to a Fano polygon with six $\frac{1}{3}(1,1)$ cones which appears in the known classification of \cite{MinimalityandMutationEquivalenceofPolygons}.

Alternatively in case (ii) when $s^{2}+s+1 \equiv 0 \modb{r}$:
\[ s^{2}+s+1 = nr, \text{ for some } n \in \Z. \]
The linear map determined by the matrix
\[ \begin{pmatrix} 1+s & r \\ -n & -s \end{pmatrix} \in \text{GL}_{2}(\Z), \]
maps $\sigma_{3}$ and $\sigma_{6}$ onto $\sigma_{4}$ and $\sigma_{1}$ respectively. Hence all the $\sigma_{i}$ represent $\frac{1}{r}(1,s)$ cyclic quotient singularities. We have shown the following proposition:

\begin{prop}\label{4.7}
There exists a Fano polygon $P$ such that 
\[  \text{SC}(P) = \left( 0, \left\{ 6 \times \frac{1}{r}(1,s) \right\} \right), \]
if and only if either $(r,s)=(3,1)$ or $p \equiv 1 \modb{6}$ for all primes $p \mid r$, and $s^{2} + s +1 \equiv 0 \modb{r}$. Furthermore this polygon is unique up to isomorphism.
\end{prop}

Propositions \ref{4.4}, \ref{4.5}, \ref{4.6} and \ref{4.7} together prove Theorem \ref{1.7}.

\begin{ex}
Consider when $(r,s)=(3,1)$. By Theorem \ref{1.7} there exists a unique polygon with singularity content $\text{SC}(P) = \left( 0, \left\{ 6 \times \frac{1}{3}(1,1) \right\} \right)$, and by Theorem \ref{1.6} a model for this polygon has vertices $(0,1), (-3,2), (-3,1), (0,-1), (3,-2)$ and $(3,-1)$. This polygon appears in \cite{MinimalityandMutationEquivalenceofPolygons} as the unique Fano polygon with singularity content of the form $\text{SC}(P) = \left( n, \left\{ 6 \times \frac{1}{3}(1,1) \right\} \right)$ where $n \in \mathbb{Z}_{\geq 0}$, and also appears in \cite{RestrictionsOnTheSingularityContentOfAFanoPolygon} as the unique Fano polygon with singularity content of the form $\text{SC}(P) = \left( 0, \left\{ k \times \frac{1}{r}(1,1) \right\} \right)$ where $k,r \in \mathbb{Z}_{>0}$.
\end{ex}

\section*{Acknowledgements}

Much of this work was completed on a trip by DC to Kyoto Sangyo University supported by JSPS Grant-in-Aid for Scientific Research (B) 18H01134. DC would like to thank Alexander Kasprzyk, his doctoral advisor, for his continued help and support. This work was partially supported by Kasprzyk's EPRSC Fellowship EP/NO22513/1 and Grant-in-Aid for Young Scientists (B) 17K14177.

\bibliographystyle{plain}

\end{document}